\documentclass{amsart}

\usepackage{color,graphicx,amssymb,latexsym,amsfonts,txfonts,amsmath,amsthm}
\usepackage{pdfsync}
\usepackage{amsmath,amscd,mathrsfs}
\usepackage[all,cmtip]{xy}

\usepackage{hyperref}
\hypersetup{
    colorlinks=true,       
    linkcolor=blue,          
    citecolor=blue,        
    filecolor=blue,      
    urlcolor=blue           
}

\theoremstyle{plain}
\newtheorem{theo}{Theorem}
\newtheorem{prop}[theo]{Proposition}
\newtheorem{lemm}[theo]{Lemma}
\newtheorem{coro}[theo]{Corollary}

\theoremstyle{definition}

\newtheorem{rema}[theo]{Remark}

\title{A structural description of extended ${\mathbb Z}_{2n}$-Schottky groups}
\author{Rub\'{e}n A. Hidalgo}
\address{Departamento de Matem\'atica y Estad\'{\i}stica, Universidad de La Frontera. Casilla 54-D, 4780000 Temuco, Chile}
\email{ruben.hidalgo@ufrontera.cl}

\date{}

\thanks{Partially supported by Projects Fondecyt 1190001 and 1220261}

\begin{document}

\begin{abstract}
Real points of Schottky space ${\mathcal S}_{g}$ are in correspondence with extended Kleinian groups $K$ containing, as a normal subgroup, a Schottky group $\Gamma$ of rank $g$ such that $K/\Gamma \cong {\mathbb Z}_{2n}$ for a suitable integer $n \geq 1$. These kind of groups are called extended ${\mathbb Z}_{2n}$-Schottky groups of rank $g$. 
 In this paper, we provide a structural decomposition theorem, in terms of Klein-Maskit's combination theorems, of these kind of groups.
\end{abstract}

\maketitle

\section{Introduction}
 A Schottky group of rank $g$ is a purely loxodromic Kleinian group, isomorphic to the free group $F_{g}$ of rank $g$, and with a non-empty region of discontinuity. In Section \ref{Sec:Schottkygroups} we recall the geometric decription of these groups. Such a geometrical description asserts that Schottky groups of rank $g \geq 1$ are a free product, in the sense of Klein-Maskit's combination theorem \cite{Maskit:Comb, Maskit:Comb4}, of $g$ cyclic groups generated by loxodromic elements. It also asserts that any two Schottky groups of the same rank are quasiconformaly conjugated.
 
 An extended Kleinian group $K$ is called an extended ${\mathbb Z}_{2n}$-Schottky group of rank $g$, where $n \geq 1$, it it contains a Schottky group $\Gamma$ of rank $g$ as a normal subgroup with $K/\Gamma \cong {\mathbb Z}_{2n}$. In the case $n=1$, these groups are also called extended Schottky groups of rank $g$.
In this paper, we provide 
a structural decomposition, in terms of Klein-Maskit's combination theorems, for the extended ${\mathbb Z}_{2n}$-Schottky groups, (see Theorem \ref{clasifica1}). For $n=1$, such a structural picture was provided in \cite{GH}.

\medskip

Extended ${\mathbb Z}_{2n}$-Schottky groups are related to the real points of Schottky space.
 If $g \geq 1$, the space that parametrizes the ${\rm PSL}_{2}({\mathbb C})$-equivalence classes of Schottky representations of $F_{g}$ (i.e., faithful homomorphisms  $\theta:F_{g} \to {\rm PSL}_{2}({\mathbb C})$ with $\theta(F_{g})$ a Schottky group)  is the 
marked Schottky space ${\mathcal MS}_{g}$. It is known that, if $\Gamma$ is a Schottky group of rank $g$, then ${\mathcal MS}_{g}$ is isomorphic to the quasiconformal deformation space ${\mathcal Q}(\Gamma)$ (see, for instance, \cite{Bers,B,Nag}); so a non-compact, connected and non-simply connected complex manifold. It is known that ${\mathcal MS}_{1}$ can be identified with ${\mathbb D}^{*}=\{z \in {\mathbb C}: 0<|z|<1\}$. 
If $g \geq 2$, then ${\mathcal MS}_{g}$ has dimension $3(g-1)$ and it is a domain of holomorphy of ${\mathbb C}^{3g-3}$ \cite{Nag}. If $g \geq 2$, then its group of holomorphic automorphisms is isomorphic to ${\rm Out}(F_{g})$ \cite{Earle} and it acts properly discontinuously on it; the quotient complex orbifold ${\mathcal S}_{g}={\mathcal MS}_{g}/{\rm Out}(F_{g})$, also of dimension $3(g-1)$, is called the Schottky space. This orbifold parametrizes the ${\rm PSL}_{2}({\mathbb C})$-conjugacy classes of Schottky groups of rank $g$.
Real structures (i.e., antiholomorphic automorphisms of order two) on ${\mathcal MS}_{g}$ were studied in \cite{HS}. All real structures of ${\mathcal MS}_{1}={\mathcal S}_{1}={\mathbb D}^{*}$ are conjugated to $j(z)=\overline{z}$, which has exactly two connected componets (the intersection of the real line with it). In ${\mathcal MS}_{2}$ there are exactly $4$ non-conjugated real structures: one without real points, one with exactly eight real components, one with three and one with four. If $g \geq 3$, then ${\mathcal MS}_{g}$ has exactly $T_{g}+1$ real structures, up to conjugation, where $T_{g}$ is the number of conjugacy classes of elements of order two in ${\rm Out}(F_{g})$.
In the same paper, it was also observed that a real point of a real structure of ${\mathcal MS}_{g}$, for $g \geq 2$,  
 can be identified with an extended Schottky group of rank $g$ and that each irreducible component of the locus of real points is a real analytic embedding of the quasiconformal deformation space of an extended Schottky group of rank $g$. 

If $g \geq 2$, then all real structures on ${\mathcal MS}_{g}$ induce the same real structure on the complex orbifold ${\mathcal S}_{g}$ (this one obtained by conjugating Schottky groups by the usual complex conjugation map $j(z)=\overline{z}$). In particular, the real points of real structures on ${\mathcal MS}_{g}$ are projected to real points of ${\mathcal S}_{g}$. Unfortunately, these are not all the real points of ${\mathcal S}_{g}$. A real point of ${\mathcal S}_{g}$ is the class of a Schottky group $\Gamma$ of rank $g$ which is ${\rm PSL}_{2}({\mathbb C})$-conjugated to $\overline{\Gamma}=j \Gamma j$, that is, there is a M\"obius transformation $A \in {\rm PSL}_{2}({\mathbb C})$ such that $\Gamma=T \Gamma T^{-1}=\Gamma$, where $T=A j$. The group $K=\langle T, \Gamma \rangle$ turns out to  be an extended ${\mathbb Z}_{2n}$-Schottky group of rank $g$, for a suitable integer $n \geq 1$.  Conversely, every  extended ${\mathbb Z}_{2n}$-Schottky group of rank $g$ defines a real point of ${\mathcal S}_{g}$. In this way, our structural description of these groups is a starting point to the study of the connectivity and the irreducible components of the locus of real points in ${\mathcal S}_{g}$.

\medskip

In terms of anticonformal automorphisms of closed Riemann surfaces, the above groups can be described as follows.
Let $S$ be a closed Riemann surface of genus $g$ and let $\tau:S \to S$ be an anticonformal automorphism of order $2n$, for some $n \geq 1$. If $n=1$, then $\tau$ is also called a symmetry (or a real structure) on $S$.  Koebe's retrosection theorem \cite{Bers,Koebe} asserts that there is a 
Schottky group $\Gamma$ of rank $g$, with region of discontinuity $\Omega$, such that $S=\Omega/\Gamma$ (we say that $S$ is uniformized by $\Gamma$). 
The planarity theorem \cite{Maskit:lowest} asserts that these correspond to the lowest uniformization of $S$. 
Sometimes it is possible to find such a Schottky group $\Gamma$ with the extra property that $\tau$ lifts, that is, there is an extended M\"obius transformation $\widehat{\tau}$ such that $\widehat{\tau} \;\Gamma \; \widehat{\tau}^{-1}=\Gamma$ (necessary and sufficient conditions, for $g \geq 2$, for the lifting to happen is the existence of a collection of pairwise disjoint simple loops which is invariant under $\tau$ and all of them cut-off $S$ into genus zero surfaces \cite{Hidalgo:SchottkyAuto, Hidalgo-Maskit}). In this case, 
$K=\langle \Gamma, \widehat{\tau} \rangle$ is an extended ${\mathbb Z}_{2n}$-Schottky group of rank $g$ and $K/\Gamma=\langle \tau \rangle$. 

\medskip

In terms of handlebodies, these groups are described as follows.
Let $M$ be a handlebody of genus $g$ and let $\tau:M \to M$ be an orientation-reversing homeomorphism of finite order $n$. It is known that there is a Schottky group $\Gamma$ of rank $g$ (so ${\mathbb H}^{3}/\Gamma$ is homeomorphic to the interior of $M$) such that $\tau$ acts as an isometry (for the induced hyperbolic metric). By lifting $\tau$ to ${\mathbb H}^{3}$, we obtain an extended M\"obius transformation $T$ (which necessarily self-conjugate $\Gamma$). In this case, $K=\langle T, \Gamma \rangle$ is an extended ${\mathbb Z}_{2n}$-Schottky group of rank $g$. Our structural description permits to describe the locus of fixed points of $\tau$ (Section \ref{handles}).

\section{Preliminaries}\label{Sec:Prelim}
In this section, we review some of the definitions, set some notations and recall some technical results we will need in this paper. Generalities on Kleinian and extended Kleinian groups can be found, for instance,  in the books \cite{Maskit:book, MT}. 

We use the symbol $\Gamma<K$ (respectively, $\Gamma \lhd K$) to say that $\Gamma$ is a subgroup (respectively, normal subgroup) of a group $K$. The composition of the maps $f$ and $h$ is as usually denoted by the symbol $f \circ h$, but if we are composing (extended) M\"obius transformations $A$ and $B$ we will use the symbol $AB$.

\subsection{Extended Kleinian groups}
Let ${\mathbb M}={\rm PSL}_{2}({\mathbb C})$ be the group of {\it M\"obius transformations} and $\widehat{\mathbb M}$ be the group generated by ${\mathbb M}$ and the complex conjugation $J(z)=\overline{z}$. A transformation in $\widehat{\mathbb M} \setminus {\mathbb M}$ is called an {\it extended M\"obius transformation}. It is well known that ${\mathbb M}$ (respectively, $\widehat{\mathbb M}$) is the full group of conformal (respectively, conformal an anticonformal) automorphisms of the Riemann sphere $\widehat{\mathbb C}$.
If $K<\widehat{\mathbb M}$, then we set $K^{+}:=K \cap {\mathbb M}$ and, when $K\neq K^{+}$, we say that $K^{+}$ is the {\it orientation-preserving half} of $K$. 

An extended M\"obius transformation whose square is an elliptic transformation is called {\it pseudo-elliptic} (if the square is the identity, then we say that it is a {\it reflection} if it has fixed points, otherwise it is called an {\it imaginary reflection}). Similarly, if the square is a loxodromic transformation (in fact a hyperbolic one), then we say that it is a {\it glide-reflection} and , if the square is parabolic, then we say that it is {\it pseudo-parabolic}.

A {\it Kleinian group} (respectively, an {\it extended Kleinian group}) is a discrete subgroup of ${\mathbb M}$ (respectively, a discrete subgroup of $\widehat{\mathbb M}$ necessarily containing extended M\"obius transformations). The {\it region of discontinuity} of a (extended) Kleinian group $K$ is the open set (which might be empty) $\Omega \subset \widehat{\mathbb C}$ consisting of those points on which $K$ acts discontinuously. The complement closed set $\Lambda=\widehat{\mathbb C} \setminus\Omega$ is called the {\it limit set} of $K$. 
If $\Lambda$ is finite (respectively, infinite), then $K$ is called {\it elementary} (respectively, {\it non-elementary}).  If $\Omega \neq \emptyset$, then we say that $K$ is of the second type.

If $K_{1}<K_{2}<\widehat{\mathbb M}$ and $K_{1}$ has finite index in $K_{2}$, then one is discrete if and only if the other is to; in which case both have the same region of discontinuity. In particular, if $K<\widehat{\mathbb M}$ contains extended M\"obius transformations, then $K$ is an extended Kleinian group if and only if  $K^{+}$ is a Kleinian group. 

A {\it function group} (respectively, an {\it extended function group}) is a Kleinian group (respectively, and extended Kleinian group) of the second type containing an invariant connected component of its region of discontinuity. Notice that if $K$ is an extended function group, then $K^{+}$ is a function group, but the converse is not in general true (but in the negative, $K^{+}$ is a quasifuchsian group and each element of $K \setminus K^{+}$ permutes the two components of $\Omega$). 

\subsection{Klein-Maskit's combination theorems}
The decomposition of function groups (respectively, extended function groups), in the sense of Klein-Maskit's combination theorems is provided in \cite{Maskit:function2, Maskit:function1, Maskit:function} (respectively,  in \cite{H:invariant}). It is a Kleinian group version of free products and HNN-extensions. Below, we state a simple version of the Klein-Maskit combination theorems which will be enough for us in this paper.

\begin{theo}[Klein-Maskit's combination theorems \cite{Maskit:Comb, Maskit:Comb4}]\label{KMC}
\mbox{}
\\
{\rm (1)} (Free products) 
For $j=1,2$, 
let $K_{j}$ be a (extended) Kleinian group with region of discontinuity $\Omega_{j}$ and let ${\mathcal F}_{j}$ be a fundamental domain for $K_{j}$. Assume that there is a simple closed loop $\Sigma$, contained in the interior of  ${\mathcal F}_{1} \cap {\mathcal F}_{2}$, bounding two discs $D_1$ and $D_2$, so that, for $j=1,2$, the set $\Sigma \cup D_{j} \subset  \Omega_{3-j}$ is precisely invariant under the identity in $K_{3-j}$. Then $K = \langle K_1, K_2\rangle$ is a (extended) Kleinian group, with fundamental domain ${\mathcal F}_{1} \cap {\mathcal F}_{2}$, which is the free product of $K_{1}$ and $K_{2}$. Every finite order element in $K$ is conjugated in $K$ to a finite order element of either $K_{1}$ or $K_{2}$. Moreover, if both $K_{1}$ and $K_{2}$ are geometrically finite, then $K$ is so.

\smallskip
\noindent
{\rm (2)} (HNN-extensions) Let $K_{0}$ be a (extended) Kleinian group with region of discontinuity $\Omega$ and let ${\mathcal F}$ be a fundamental domain for $K_{0}$. Assume that there are two pairwise disjoint simple closed loops $\Sigma_{1}$ and $\Sigma_{2}$, both of them contained in the interior of  ${\mathcal F}_{0}$, so that $\Sigma_{j}$ bounds a disc $D_{j}$ such that $(\Sigma_{1} \cup D_{1}) \cap (\Sigma_{2} \cup D_{2})=\emptyset$ and with $\Sigma_{j} \cup D_{j} \subset  \Omega$ precisely invariant under the identity in $K_{0}$. If $T$ is either a loxodromic transformation or a glide-reflection so that $T(\Sigma_{1})=\Sigma_{2}$ and $T(D_{1}) \cap D_{2}=\emptyset$, then $K = \langle K_{0}, f \rangle$ is a (extended) Kleinian group, with fundamental domain ${\mathcal F}_{1} \cap (D_{1} \cup D_{2})^{c}$, which is the HNN-extension of $K_{0}$ by the cyclic group $\langle T \rangle$.  Every finite order element of $K$ is conjugated in $K$ to a finite order element of $K_{0}$. Moreover, if $K_{0}$ is geometrically finite, then $K$ is so.
\end{theo}

\subsection{Schottky groups: a geometrical picture}\label{Sec:Schottkygroups}
The {\it Schottky group of rank $0$} is just the trivial group. A {\it Schottky
group of rank $g \geq 1$} is a Kleinian group $\Gamma$ generated by loxodromic
transformations $A_1,\ldots,A_g$, so that there are $2g$ pairwise disjoint simple loops,
$\Sigma_{1},\Sigma'_{1},\ldots,\Sigma_{g},\Sigma'_{g}$, bounding a $2g$-connected domain ${\mathcal D}\subset \widehat{\mathbb C}$, where $A_i(\Sigma_i)=\Sigma'_i$, and $A_i({\mathcal D})\cap {\mathcal D}=\emptyset$, for $i=1,\ldots,g$. The collection of loops $\Sigma_{1}$, $\Sigma'_{1}$,..., $\Sigma_{g}$ and $\Sigma'_{g}$ is called a {\it fundamental set of loops}  for $\Gamma$ with respect to the above generators (these groups are constructed, using part (1) in Klein-Maskit's combination theorem, as the free product of cyclic loxodromic groups). The region of discontinuity $\Omega$ of a Schottky group $\Gamma$ of rank $g$ is known to be  connected and dense in $\widehat{\mathbb C}$ and  that $S=\Omega/\Gamma$ is a closed Riemann surface of genus $g$ (Koebe's retrosection theorem \cite{Bers,Koebe} states that, up to conformal isomorphism, every closed Riemann surface is obtained in this way). 

\begin{rema}
A Schottky group of rank $g$ can be defined as a purely loxodromic Kleinian group of the second kind which
is isomorphic to a free of rank $g$ \cite{Maskit:Schottky groups}.  Also, it can be defined as a purely loxodromic and geometrically finite Kleinian group, isomorphic to a free group. The geometrical definition permits to observe that any two Schottky groups of same rank are topologically conjugated.
\end{rema}

\subsection{Uniformizations}
Let $S$ be a closed Riemann surface.
A triple $(\Delta,\Gamma,P)$ is called an {\it uniformization} of $S$  if $\Gamma$ is a Kleinian group, $\Delta$ is a $\Gamma$-invariant connected component of its region of discontinuity and $P:\Delta \to S$ is a regular planar covering with $\Gamma$ as its group of deck transformations. In \cite{Maskit:lowest}, Maskit provided a description of all uniformizations (regular planar coverings) in terms of a (not unique) collection of pairwise disjoint loops. The collection of uniformizations of $S$ is partially ordered in the sense that $(\Delta_{1},\Gamma_{1},P_{1})$ is higher than $(\Delta_{2},\Gamma_{2},P_{2})$ if there is a covering map $Q:\Delta_{1} \to \Delta_{1}$ so that $P_{1}=P_{2} \circ Q$. Let us consider an uniformization $(\Delta,\Gamma,P)$. It is a highest  one (with respect to the previous partial ordering)  if and only if $\Delta$ is simply-connected. By the results in \cite{Maskit:lowest,Tamura}, it is a lowest one if $\Gamma$ is a Schottky group of rank $g$ (in which case $\Delta$ is the region of discontinuity of $\Gamma$); we call it a  {\it Schottky uniformization} of $S$. 

\subsection{A lifting property}
A conformal (respectively, anticonformal) automorphism $\tau$ of $S$ {\it lifts} with respect to an uniformization
$(\Delta,\Gamma,P)$ if there is an automorphism $\kappa$ of $\Delta$ so that $P \circ \kappa=\tau \circ P$.  A group $H$ of (conformal/anticonformal) automorphisms of $S$ {\it lifts} with respect to the above uniformization if and only if each of its elements lifts.

\begin{rema}
Let $(\Delta,\Gamma,P)$ be a uniformization of $S$ and let $H<Aut(S)$ be a group that lifts. A lifting $k \in {\rm Aut}(\Delta)$ of $h \in H$ is not required to be the restriction of a (extended) M\"obius transformation. If $K$ is the group generated by all these liftings, then $K<Aut(\Delta)$, $\Gamma \lhd K$ and $K/\Gamma \cong H$, but, in general, $K$ is not a group of (extended) M\"obius transformations. Now, if the above is a Schottky uniformization (so $\Delta=\Omega$, the region of discontinuity of $\Gamma$), then 
$\Omega$ is of class $O_{AD}$ (that is, it admits no holomorphic function with finite Dirichlet norm \cite[pg 241]{A-S}). It follows that a conformal (respectively, anticonformal) automorphism of $\Omega$ is the restriction of a M\"obius (respectively, extended M\"obius) transformation. It follows that $K$ is a (extended) Kleinian group $\Gamma$ as a finite index normal subgroup with $K/\Gamma \cong H$. 
\end{rema}

\begin{rema}
In \cite{H-cota} it was noted that if $H$ lifts with respect to a uniformization of $S$ which is not a highest one, then the order of $H$ is at most $24(g-1)$ (and if $H$ only consists of conformal automorphisms, then its order is at most $12(g-1)$).
\end{rema}

If $\Delta$ is simply-connected, clearly every automorphism of $S$ lifts.  If $\Delta$ is not simply-connected, it may be that some automorphism of $S$ does not lift to an automorphism of $\Delta$. The following result provides necessary and sufficient conditions for the lifting property to hold in the case of Schottky uniformizations.

\begin{theo}\cite{Hidalgo:SchottkyAuto,Hidalgo-Maskit}\label{teo1}
Let $(\Omega,\Gamma,P)$ be a Schottky uniformization of the closed Riemann surface $S$ of genus $g \geq 2$.
Let $H$ be a group of automorphisms (conformal/anticonformal) of $S$. Then,  $H$ lifts with respect to the above Schottky uniformization if and only if
there is a collection $\mathcal F$ of pairwise disjoint simple loops on $S$ such that:
\begin{itemize}
\item[(i)] each connected component of $S \setminus \mathcal F$ is a planar surface; 
\item[(ii)] $\mathcal F$ is invariant under the action of $H$; and
\item[(iii)] for each $\alpha \in {\mathcal F}$, $P^{-1}(\alpha)$ is a collection of pairwise disjoint simple loops in $\Omega$.
\end{itemize}
\end{theo}

A collection of loops ${\mathcal F}$, as in Theorem \ref{teo1}, which is minimal
(in the sense that by deleting a non-empty sub-collection from ${\mathcal F}$, then one of the above three properties fails)  will be called a {\it fundamental collection of loops} associated to the pair  $\{(\Omega,G,P),H\}$.

\begin{rema}
We should note that Theorem \ref{teo1} can be seen a consequence of the Equivariant Loop Theorem \cite{Y-M}, whose
proof is  based on minimal surfaces, that is, surfaces that minimize locally the area. The proof given in
\cite{Hidalgo:SchottkyAuto} only uses arguments proper to (planar) Kleinian groups and the hyperbolic metric.
\end{rema}

\subsection{Hyperbolic extensions}

Each M\"obius (respectively, extended M\"obius) transformation acts (by Poincar\'e's extension) as an orientation-preserving (respectively, orientation-reversing) isometry of the hyperbolic $3$-space ${\mathbb H}^{3}=\{(z,t) \in {\mathbb C} \times {\mathbb R}: t>0\}$ with the hyperbolic metric $ds^{2}=(|dz|^2+dt^2)/t^2$. If $\Gamma$ is a (extended) Kleinian group, then ${\mathcal O}^{3}_{\Gamma}={\mathbb H}^{3}/\Gamma$ is a $3$-dimensional hyperbolic orbifold and the $2$-dimensional orbifold ${\mathcal O}^{2}_{\Gamma}=\Omega/\Gamma$ is its conformal boundary. In the case that $\Gamma$ is a torsion free Kleinian group, then ${\mathcal O}^{3}_{\Gamma}$ is a hyperbolic $3$-manifold and ${\mathcal O}_{\Gamma}^{2}$ a Riemann surface. 

If $\Gamma$ is a Schottky group of rank $g$, then ${\mathcal O}_{\Gamma}^{3}$ is homeomorphic to the interior of a handlebody of genus $g$ and its conformal boundary ${\mathcal O}_{\Gamma}^{2}$ is a closed Riemann surface of genus $g$. 

Every torsion free Kleinian group $\Gamma$, for which ${\mathcal O}_{\Gamma}^{3}$ is homeomorphic to the interior of a handlebody of genus $g$, is a Kleinian group isomorphic to the free group $F_{g}$. The conformal boundary $\Omega/\Gamma$ coincides with the topological boundary if and only if $\Gamma$ is a Schottky group of rank $g$. This is equivalent to say that $\Gamma$ is geometrically finite (i.e., it has a finite sided fundamental polyhedron) and that ${\mathbb H}^{3}/\Gamma$ has injectivity radius bounded away from zero (this is equivalent to the purely loxodromic property).

\section{Decomposition structure of extended ${\mathbb Z}_{2n}$-Schottky groups}
Let us recall that an extended ${\mathbb Z}_{2n}$-Schottky group of rank $g$ is an extended Kleinian group $K$ containing as a normal subgroup a Schottky group $\Gamma$ of rank $g$ such that $K/\Gamma \cong {\mathbb Z}_{2n}$.

In this section, we provide our main result, a general structure picture, in terms of Klein-Maskit's combination theorems, of extended ${\mathbb Z}_{2n}$-Schottky groups. In Section \ref{Sec:basico}, we define the basic groups of type $n$. In Section \ref{Sec:general} we provide a necessary and sufficient condition for a group obtained by applying the Klein-Maskit's combination theorem on them to be an extended ${\mathbb Z}_{2n}$-Schottky group. Finally, in Section \ref{Sec:picture}, we provide the main result that states that all of these extended groups are so obtained. So, from now on, we fix an integer $n \geq 1$.

\subsection{Basic groups of type $n$}\label{Sec:basico} 
To state our structural description of the extended ${\mathbb Z}_{2n}$-Schottky groups, we will need some elementary (extended) Kleinian groups (see Figure \ref{grafica1}), which we will call the {\it basic groups of type $n$}.
\begin{itemize} 
\item[(T0).-] Cyclic groups generated by a loxodromic transformation.

\item[(T1).-] Cyclic groups generated by a glide-reflection transformation.

\item[(T2).-] Cyclic groups generated by elliptic transformation of order a divisor of $n$.

\item[(T3).-] Cyclic groups generated by a pseudo-elliptic transformation of order $2d$, where $d$ is a divisor of $n$, but $2d$ is not a divisor of $n$.

\item[(T4).-] Abelian groups generated by a loxodromic transformation and an elliptic transformation of order a divisor of $n$ (in particular, both fixed points of the elliptic are the same as for the loxodromic).

\item[(T5).-] Groups generated by a loxodromic transformation $A$ and a pseudo-elliptic transformation $B$, of order a divisor of $2n$ but not of $n$, so that $B^{-1}ABA=I$ (in particular, both fixed points of the loxodromic transformation are permuted by the pseudo-elliptic transformation).

\item[(T6).-] If $n$ is even, groups generated by a glide-reflection transformation $A$ and an elliptic transformation $B$ of order $2$ so that $BA=AB=I$ (in particular, both fixed points of the elliptic are the same as for the glide-reflection).

\item[(T7).-] If $n$ is odd, cyclic groups of order $2$ generated by reflections.

\item[(T8).-] If $n$ is odd, groups generated by the reflection of a circle $\Sigma$ and a 
discrete group $F$ (of orientation-preserving conformal automorphisms) keeping invariant $\Sigma$ and 
so that $\Omega_{F}/F$ (where $\Omega_{F}$ is the region of discontinuity of $F$) is a connected Riemann orbifold whose conical points have orders divisors of  $n$. 

\end{itemize}

\begin{figure}
\centering
\includegraphics[width=9cm]{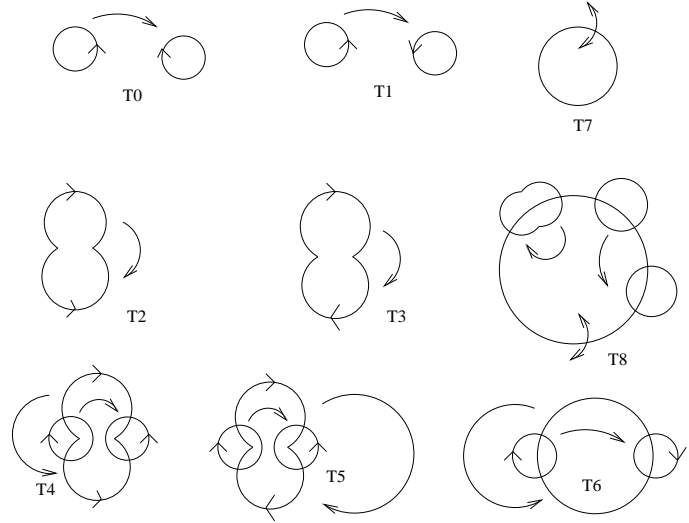}
\caption{Basic groups of type $n$}
\label{grafica1}
\end{figure}

\subsection{General groups of type $n$}\label{Sec:general}
A group, constructed by use of Klein-Maskit's combination theorem (1), by using the basic groups of type $n$,  will be called a {\it general group of type $n$}. 
Those basic groups of type $n$ which are extended ${\mathbb Z}_{2m}$-Schottky groups, where $m$ is divisor of $n$ (necessarily of rank $0$ and $1$) are given by (T1), (T3), (T5), (T6), (T7) and (T8). 
There are general groups of type $n$ which are not extended ${\mathbb Z}_{2n}$-Schottky groups.

If $K_{*}$ is a general group, then, as consequence of Klein-Maskit's combination theorems, the following properties can be observed.

\begin{enumerate}
\item $K_{*}$ is a discrete group;
\item the limit set of $K_{*}$ is a Cantor set;
\item $K_{*}^{+}$ contains no parabolic transformations;
\item every non-loxodromic transformation in $K_{*}^{+}$ is either the identity or conjugated to a power of some elliptic or pseudo-elliptic generator used in the basic groups construction.
\end{enumerate}

The following gives a necessary and sufficient condition for a general group of type $n$ to be an extended ${\mathbb Z}_{2n}$-Schottky group.

\begin{prop} \label{propo1}
A general group $K_{*}$ is a extended ${\mathbb Z}_{2n}$-Schottky group  if and only if 
\begin{enumerate}
\item $K_{*}$ contains orientation-reversing transformations, and 
\item there is a surjective homomorphism
$$\Phi:K_{*} \to {\mathbb Z}_{2n}=\langle x : x^{2n}=1\rangle$$ 
with torsion free kernel containing only orientation-preserving transformations. 
\end{enumerate}
\end{prop}
\begin{proof}
One direction is clear. The other one is consequence of the following. If $\Gamma_{*}$ is the kernel of $\Phi$ as required, then it is a function group containing only loxodromic transformations and whose limit set is a Cantor set. Then, as consequence of the classification of function groups \cite{Maskit:function} $\Gamma_{*}$, is a Schottky group.
\end{proof}

\subsection{Structural decomposition of extended ${\mathbb Z}_{2n}$-Schottky groups}\label{Sec:picture}
In Proposition \ref{propo1} we have seen that certain general groups of type $n$ are in fact extended ${\mathbb Z}_{2n}$-Schottky groups. The following result states the converse, that is, that every extended ${\mathbb Z}_{2n}$-Schottky groups is a general group of type $n$.

\begin{theo}[Decomposition theorem for extended ${\mathbb Z}_{2n}$-Schottky groups]\label{clasifica1}  
\mbox{}
\begin{enumerate}
\item Every extended ${\mathbb Z}_{2n}$-Schottky group is a general group of type $n$, where at least one of the basic groups of type (T1), (T3), (T5), (T6), (T7) or (T8) is used in its construction.

\item A general group of type $n$ an extended ${\mathbb Z}_{2n}$-Schottky group if and only if both conditions below are satisfied.
\begin{enumerate}
\item In the construction there appears at least one of the types (T1), (T3), (T5), (T6), (T7) or (T8).

\item If $n \geq 2$ and there are not groups of types (T0), (T1) and (T6) and also no groups of type (T8) containing a glide-reflection, then the greater common divisor of all the values of the form $2n/r$, were $r$ runs over all orders of elliptic and pseudo-elliptic transformations used in the basic groups of type (T2), (T3), (T4), (T5), (T7) and (T8), is $1$.

\end{enumerate}
\end{enumerate}
\end{theo}

The particular case $n=1$ was previously obtained in \cite{GH}.

\section{Proof of Theorem \ref{clasifica1}}\label{clase}
The cases $g=0,1$ correspond to the basic groups. So, we only need to take care of the case $g \geq 2$.
Let $K$ be a extended ${\mathbb Z}_{2n}$-extended Schottky group of rank $g \geq 2$ and $\Gamma \lhd K$ be a Schottky group of rank $g$ so that $H=K/\Gamma=\langle \tau \rangle \cong {\mathbb Z}_{2n}$. The Riemann surface $S=\Omega/\Gamma$ admits the anticonformal automorphism $\tau$ of order $2n$ which lifts to the region of discontinuity $\Omega$ as an extended M\"obius transformation $\widehat{\tau}$ which normalizes $\Gamma$, $\widehat{\tau}^{2n} \in \Gamma$  and $\widehat{\tau}^{j} \notin \Gamma$, for $j=1,...,2n-1$. Moreover, $K=\langle \widehat{\tau},\Gamma\rangle$. Let $\pi:S \to S/H$ be a (branched) Galois di-analytic covering with deck group $H$. We will follow the same ideas (with the corresponding modifications) as in \cite{GH} for the case $n=1$.

\subsection{Structural regions and loops}
Let us consider a Galois covering $P:\Omega \to S$ whose deck group is $\Gamma$.
As $H<Aut(S)$ lifts, with respect to the Schottky uniformization $(\Omega,\Gamma,P)$, to obtain $K$, then (by Theorem \ref{teo1}) there is a fundamental collection ${\mathcal F}$ associated to the pair  $\{(\Omega,\Gamma,P),H\}$. Let us denote by $\widetilde{\mathcal F}$ the collection of pairwise disjoint simple loops in $\Omega$ obtained by the lifting of the loops in ${\mathcal F}$ under $P$. Recall that such a collection of loops is minimal, that is, by deleting any subcollection from it, the resulting collection is not longer a fundamental collection for the above uniformization.
Each loop $\beta \in \widetilde{\mathcal F}$ is called a {\it structure loop} for $K$ and each of the connected components $R$ in $\Omega \setminus \widetilde{\mathcal F}$ is called a {\it structure region} for $K$. We denote by $K_{R}$ and by $K_{\beta}$ the corresponding $K$-stabilizers of $R$ and $\beta$.

\begin{lemm}\label{stabilizers}
\mbox{}
\begin{enumerate}
\item If $R$ is a structural region, then $K_{R}$ is either (i) trivial or (ii) a cyclic group generated by a power of $\tau$.

\item If $\beta$ is a structural loop, then $K_{\beta}$ is either:
\begin{enumerate}
\item trivial;
\item $\langle \tau^{n} \rangle$, where $\tau^{n}$ permutes both structure regions with common boundary; 
\item a cyclic group of order a divisor of $n$ (if preserves orientation) or $2n$ (if reverses the orientation), keeping invariant each one of the two structure regions with $\beta$ in the boundary.

\end{enumerate}
\end{enumerate}
\end{lemm}
\begin{proof}
 If $A \subset S \setminus {\mathcal F}$ is any connected component and $R \subset \Omega \setminus \widetilde{\mathcal F}$ is any connected component such that $P(R)=A$, then $P:R \to A$ is a homeomorphism. This asserts that there is natural isomorphism (induced by $P$) between $K_{R}$ and the $H$-stabilizer of $A$. This is particular asserts that the $K_{R}$ is isomorphic to a subgroups of $H$, that is, either (i) trivial or (ii) a finite cyclic group generated by some power of $\tau$.

Similarly, if $P(\beta)=L \in {\mathcal F}$, then $P$ induces a natural isomorphism between $K_{\beta}$ and the $H$-stabilizer of $L$. In this way, $K_{\beta}$ is either (i) trivial or (ii) a cyclic group generated by an order two element permuting both structure regions around $\beta$ (in this case, $K_{\beta}$ must be generated by the unique involution in $H$)
or (iii) a cyclic group of order a divisor of $n$ (if preserves orientation) or $2n$ (if reverses the orientation), keeping invariant each one of the two structure regions with $\beta$ in the boundary.
\end{proof}

\begin{rema}
Note, in case (b) above, that the involution $\tau^{n}$ might be (i) an elliptic transformation of order two with its both fixed points on $\beta$ (this for $n$ even) or (ii) a reflection whose circle of fixed points is $\beta$ or (iii) an imaginary reflection (both case for $n$ odd).
\end{rema}

\begin{prop}\label{equivalente}
Let $R$ and $R'$ be any two different structure regions with a common boundary loop $\beta \in \widetilde{\mathcal F}$. Then, they are
$K$-equivalent if and only if either:
\begin{itemize}
\item[({i})] $K_{\beta}$ contains an element (necessarily of order two) which does not belong
to $K_{R}$ or
\item[({ii})] there is another boundary loop $\beta' \in \widetilde{\mathcal F}$ of $R$ and an element $k \in
K \setminus K_{R}$ such that $k(\beta')=\beta$ (necessarily a loxodromic or pseudo-hyperbolic transformation).
\end{itemize}
\end{prop}
\begin{proof}
If two structure regions share a boundary structure loop, then they are $K$-equivalent if there is some element of $K$ sending a boundary loop of one to a boundary loop of the other. 
\end{proof}

\begin{prop}\label{trivial}
Let $R$ be a structure region with $K_{R}$ either trivial or a cyclic group generated by a reflection.
If $\beta$ a boundary loop of $R$ such that $K_{R} \cap K_{\beta}=\{I\}$, then there is
a non-trivial element $k \in K \setminus K_{R}$ so that $k(\beta)$ still a boundary loop of $R$.
\end{prop}
\begin{proof}
If $K_{\beta}$ contains an element outside $K_{R}$, then we are done.
Assume $K_{\beta}$ is trivial. In this case, $P(\beta) \in {\mathcal F}$ is a simple loop with trivial $H$-stabilizer. 
We have that $\beta$ is free homotopic to the product of the
other boundary loops of $R$. If none of the other boundary loops of $R$ is equivalent to $\beta$ under
$K$, then we may delete $P(\beta)$ and its $H$-translates from ${\mathcal F}$,  contradicting the
minimality of ${\mathcal F}$.
\end{proof}

If $R$ is a structural region, then $K^{+}_{R}$ is either trivial or a cyclic group generated by some elliptic transformation.

\begin{prop}\label{prop413}
 Let $R$ be a structure region with $K_{R}^{+}$ being non-trivial.  If the generator of the cyclic group $K_{R}^{+}$ has one of its fixed points 
in $R$, then the other fixed point also lie in $R$.
\end{prop}
\begin{proof}
Let $K^{+}_{R}=\langle h \rangle$ and suppose only one fixed point of $h$ belongs to $R$.
Then there is a unique structure loop $\beta$ on the boundary of $R$ stabilized by $h$. Every other structure
loop, on the boundary of $R$, has $K^{+}_{R}$-stabilizer the identity.

If $K_{R}=K_{R}^{+}$, then it follows that if were to fill in the discs bounded by the other structure loops on
the boundary of $R$, then $\beta$ would be contractible; that is, if we delete $P(\beta)$ and their
$H$-translates from ${\mathcal F}$, then we will still have a system of loops which is $H$-invariant and cut-off $S$ into genus zero surfaces, a contradiction to the minimality of  our collection ${\mathcal F}$.

If $K_{R} \neq K_{R}^{+}$, then there is some pseudo-elliptic $t \in K$ such that $K_{R}=\langle t \rangle$ and $t^{2}=h$ (in particular, $n$ must be even). If $t$ permutes both fixed points of $h$, we obtain that both fixed points of $h$ must belong to $R$. If $t$ does not permutes them, then (as $t$ also fixes the fixed points  of $h$) $t$ has to be a reflection, from which $h=t^{2}=I$, a contradiction.  
\end{proof}

The previous result asserts that a non-trivial elliptic transformation in $K_{R}^{+}$ either has both fixed points on the structure region $R$ or none of them belong to it. In the last case, there are (exactly) two boundary structure loops of $R$, each one invariant under such an elliptic transformation.

\begin{prop}\label{prop414}
Let $R$ be a structure region with non-trivial $K_{R}^{+}$.
If $\beta_{1}, \beta_{2}$ are two different boundary loops of $R$ which are invariant
under $K_{R}^{+}$, then there is a (non-trivial) element $k \in K$ so that $k(\beta_{1})=\beta_{2}$ (such an element is either loxodromic or pseudo-hyperbolic).
\end{prop}
\begin{proof}
Let us assume that there is no such element of $K$ as desired and let $R_{*}$ be the other structure region sharing $\beta_{2}$ in its boundary. Proposition \ref{prop413} asserts that on the region $R_{*}$ there is another boundary loop $\beta_{3}$ which is also invariant under $K_{R}^{+}$. All other boundary loops of $R \cup R_{*}$ (with the exception of $\beta_{1}$, $\beta_{2}$ and $\beta_{3}$) have trivial $K$-stabilizers. In particular, they are not $K$-equivalents to $\beta_{1}$, $\beta_{2}$ and $\beta_{3}$. Also, $\beta_{2}$ is neither $K$-equivalent to $\beta_{1}$ and $\beta_{3}$ (by our assumption). If we project the region $R \cup R_{*} \cup \beta_{2}$ on $S$, we obtain an homeomorphic copy and the projected loop from $\beta_{2}$ is not $H$-equivalent to none of its boundary loops. In particular, we may delete it (and its $H$-translates) obtaining a contradiction to the minimality of ${\mathcal F}$.
\end{proof}

\subsection{Structure regions with trivial stabilizers}
Let $R$ be a structure region with trivial stabilizer
$K_{R}=\{I\}$. As consequence of Proposition \ref{trivial}, every other structure region is necessarily $K$-equivalent
to $R$. It follows that $R$ is a fundamental domain for $K$ and the boundary loops are paired by
either reflections, imaginary reflections, loxodromic transformations or pseudo-hyperbolics. In this case we obtain that $K$ is the free product, by the Klein-Maskit combination theorem, of groups of types (T0), (T1), (T2) (generated by an elliptic transformation of order two), (T3) (generated by imaginary reflections) and (T7). We are done in this case.

\subsection{Structure regions with non-trivial stabilizers}
Let us now assume there is no structure region with trivial $K$-stabilizer.
Proposition \ref{equivalente}, and the fact that $S/H$ is compact and connected, permits us to construct a
connected set $\widehat{R}$ obtained as the union of a finite collection of $K$-non-equivalent structure regions
(each of them has non-trivial $K$-stabilizer) together their boundary structure loops.

\smallskip
\noindent
{\it A suitable modification of $\widehat{R}$.}
We proceed to modify $\widehat{R}$ as follows. Let $\beta$ be a structure loop in the boundary of $\widehat{R}$ and $R \subset \widehat{R}$ be the structure region with $\beta$ in its border.
Assume there is a reflection $t \in K_{R}$ keeping invariant $\beta$ (so its circle of fixed points intersects $\beta$ at two points). This situation only may happen for $n$ odd and $t=\tau^{n}$. We should also note that $K_{\beta}=\langle t\rangle$.

We know that
there is some $k \in K$ such that $k(\beta)=\beta'$ is a boundary loop of $\widehat{R}$. If $\beta' \neq \beta$ (so $k$ is either loxodromic or pseudo-hyperbolic), let
$R' \subset \widehat{R}$ be the structure region containing $\beta'$ in its border. The reflection $t'=k t k^{-1}$ belongs to $K_{R'}$ and keeps invariant $\beta'$. In the case that $t' \neq t$, we eliminate $R'$ from $\widehat{R}$ and we add to it the structure region $k^{-1}(R')$. Under this type of process, we may now assume that $t'=t$.

\smallskip
The above permits us to assume that our set $\widehat{R}$ has the following extra property: {\it if $\beta$ and $\beta'$ are boundary loops of $\widehat{R}$ such that ({i}) there is some $k \in K$ with $k(\beta)=\beta'$ and ({ii}) there is a reflection $t=\tau^{n} \in K_{R}$ keeping invariant $\beta$, then $t$ also keeps invariant $\beta'$.}

\smallskip
\noindent
{\it Step 1: Internal structural loops and amalgamated free products.}
Note that if $\beta$ is a structure loop contained in the interior of $\widehat{R}$, then there are two different structure regions $R_{1}, R_{2} \subset \widehat{R}$ with $\beta$ as their common boundary. In this case, as $R_{1}$ and $R_{2}$ are non-$K$-equivalent, $K_{\beta}=K_{R_{1}} \cap K_{R_{2}}$. By Propositions \ref{prop413} and \ref{prop414} $K_{\beta}^{+}$ is trivial. So, either $K_{\beta}$ is trivial or generated by a reflection with exactly two fixed points on $\beta$.
We now perform the amalgamated free product of $K_{R_{1}}$ and $K_{R_{2}}$ with amalgamation at $K_{R_{1}} \cap K_{R_{2}}$. In the case $K_{\beta}$ is trivial, we are obtaining free product of groups of the types (T1)--(T7) and,  in the other case, we are constructing parts of the basic groups of type (T8).

\smallskip
\noindent
{\it Step 2: Boundary structural loops and HNN-extensions.}
Next, let $\beta$ be a structure loops in the boundary of $\widehat{R}$ and let $R\subset \widehat{R}$ be the structure region with $\beta$ in its boundary.

If $K_{\beta} \cap K_{R}$ is trivial, then Proposition \ref{trivial} asserts the existence of another structure boundary loop $\beta'$ (in the boundary of $\widehat{R})$ and some $k \in K$
such that $k(\beta)=\beta'$. If $\beta' \neq \beta$, then $k$ is either loxodromic or a pseudo-hyperbolic and if $\beta'=\beta$, then $k$ is either a reflection, an imaginary reflection or an elliptic involution. In this case we obtain (by HNN extensions in the sense of the Klein-Maskit combination theorem) groups of types (T0)--(T7) as described in the theorem.

Let us now assume $K_{\beta} \cap K_{R}$ is non-trivial.

\smallskip
\noindent
(A) Assume $K_{R}=\langle t=\tau^{n} \rangle$, where $t$ is a reflection with  circle of fixed points $C_{t}$. As $K_{\beta} \cap K_{R}$ is non-trivial,
$C_{t}$ intersects $\beta$ (at exactly two points), so either:
\begin{itemize}
\item[({i})] there is an involution $k \in K$ (conformal or anticonformal) with $k(\beta)=\beta$ and
$k(\widehat{R}) \cap \widehat{R}=\beta$; or
\item[({ii})] there is another boundary loop
$\beta'$ of $\widehat{R}$ and an element $\kappa \in K$ (which is either loxodromic or a pseudo-hyperbolic) so that $\kappa(\beta)=\beta'$ and
$\kappa(\widehat{R}) \cap \widehat{R}=\beta'$. We may assume $\kappa$ to be loxodromic. In fact, if $\kappa$ is a pseudo-hyperbolic, then $\tau^{n}\kappa$ is loxodromic with the same property.
\end{itemize}

In both cases above, we may perform the HNN-extension (in the sense of Klein-Maskit combination theorem) to produce factors of type (T8).

\smallskip
\noindent
(B) Assume $K_{R}=K_{R}^{+}\langle \phi \rangle$, where $\phi=\tau^{2l}$ is an elliptic transformation. In this case we have two possibilities: either ({i}) both fixed
points of $\phi$ belong to $R$ or ({ii}) there are two boundary loops $\beta_{1}$ and $\beta_{2}$ of $R$, each one invariant under $\phi$, and there is some $k \in K$ with $k(\beta_{1})=\beta_{2}$, $k(\widehat{R}) \cap \widehat{R}=\beta_{2}$. As $K_{\beta} \cap K_{R}$ is non-trivial, necessarily
$\beta=\beta_{1}$ (or $\beta_{2}$), then (again performing HNN-extension by the Klein-Maskit combination theorem) we obtain a group of type (T8).

\smallskip
Summarizing all the above is the following, which provides the first part of the theorem.

\begin{prop}\label{clasifica0}
The extended ${\mathbb Z}_{n}$-Schottky group $K$ is a general group of type $n$, where at least one of the basic groups of type (T1), (T3), (T5), (T6), (T7) or (T8) is used in its construction.\end{prop}

\subsection{Proof of Part (2)}
Condition (1) in Proposition \ref{propo1} is equivalent to say that in the construction of $K$ we need to use  at least one of the  basic groups of types (T1), (T3), (T5), (T6), (T7) or (T8).

Once we have condition (1) full-filled, condition (2), of the same proposition, is trivial if we have used in the construction basic groups of types  (T0), (T1) or (T6) (for $n$ even). Now, if we don't use any of these three types of groups, then in order for to have condition (2) full-filled, we need  that the maximum common divisor of all values of the form $2n/r$, were $r$ runs over all orders of elliptic and pseudo-elliptic transformations used in the basic groups, is $1$.

\medskip

All the above complete the proof of Theorem \ref{clasifica1}.

\section{Example: Extended ${\mathbb Z}_{4}$-Schottky groups}\label{cason=2}
As already observed, the case $n=1$ was described in \cite{GH}.
In this section, as an example, we writte down Theorem \ref{clasifica1} for the 
next case $n=2$.

\begin{coro}\label{clasifica2}
\mbox{}
\begin{itemize}
\item[I.-] Every extended ${\mathbb Z}_{4}$-Schottky groups of rank $g$ is a general group of type $2$, where the following basic groups (of type $2$) are used.

\begin{itemize} 
\item[(T1).-] cyclic group generated by a glide-reflection transformation; 

\item[(T2).-] cyclic group generated by an elliptic transformation of order $2$;

\item[(T3).-] cyclic group generated by a pseudo-elliptic transformation of order $4$;

\item[(T4).-] Abelian group generated by a loxodromic transformation and an elliptic transformation of order $2$;

\item[(T5).-] a group generated by a loxodromic transformation $A$ and a pseudo-elliptic transformation $B$ of order $4$ so that $B^{-1}ABA=I$;

\item[(T6).-] a group generated by a glide-reflection transformation $A$ and an elliptic transformation $B$ of order $2$ so that $BA=AB=I$.

\end{itemize}

\item[II.-] A general group $K$ of type $2$, constructed using $``a_{j}"$ groups of type (Tj) (where $j=1,...,6$),
is a extended ${\mathbb Z}_{4}$-Schottky group of rank $g$ if and only if the following two conditions are satisfied:
\begin{itemize}
\item[1.-]  $a_{1}+a_{3}+a_{5}+a_{6}>0$; and

\item[2.-] $g=4a_{1}+2a_{2}+3a_{3}+4a_{4}+4a_{5}+4a_{6} -3. $
\end{itemize}

In this case, we say that $(a_1, a_2,a_3,a_4,a_5,a_6)$ is the signature of the general group $K$.

\end{itemize}
\end{coro}

\begin{rema}
Note that in the above result we have left aside the basic groups of type (T0). The reason is that, 
as $a_{1}+a_{3}+a_{5}+a_{6}>0$, if $a_{0}>0$, then we may replace the $a_{0}$ groups of type (T0) by groups of type (T1). 
\end{rema}

The number of different signatures of ${\mathbb Z}_{4}$-Schottky groups of the same rank $g$ produce topologically non-equivalent extended ${\mathbb Z}_{4}$-Schottky groups. In this way, the number of different topological extended ${\mathbb Z}_{4}$-Schottky groups of rank $g$ is the cardinality of the set
{\small
$$X_{g}=\left\{(a_{1},\ldots,a_{6}): a_{j} \in {\mathbb N}_{0}, \; a_{1}+a_{3}+a_{5}+a_{6}>0, \; g+3=4a_{1}+2a_{2}+3a_{3}+4a_{4}+4a_{5}+4a_{6} \right\},$$
}
where ${\mathbb N}_{0}$ denotes the set of  non-negative integers.

If we set 
$$N_{g}=\left\{(\alpha,\beta,\gamma,\delta): \alpha, \beta, \gamma, \delta \in {\mathbb N}_{0}, \; g+3=4\alpha+2\beta+3\gamma+4\delta, \; \alpha+\gamma>0 \right\},$$
and
$$Y(\alpha)=\left\{(x,y,z): x,y,z \in {\mathbb N}_{0}, \; x+y+z=\alpha\right\}.
$$
then  (by setting $\alpha=a_{1}+a_{5}+a_{6}$, $\beta=a_{2}$, $\gamma=a_{5}$ and $\delta=a_{4}$),
$$\#X_{g}=\sum_{(\alpha,\beta,\gamma,\delta) \in N_{g}} \#Y(\alpha)=\frac{1}{2}\sum_{(\alpha,\beta,\gamma,\delta) \in N_{g}} (1+\alpha)(2+\alpha).$$

If we set 
$$N_{g}(\alpha)=\left\{(\beta,\gamma,\delta): \beta, \gamma, \delta \in {\mathbb N}_{0}, \; g+3=4\alpha+2\beta+3\gamma+4\delta, \; \alpha+\gamma>0  \right\},$$
then (since $\alpha \in \left\{0,1,\ldots, \left[\frac{g+3}{4}\right]\right\}$)
$$\#X_{g}=\frac{1}{2}\sum_{\alpha=0}^{\left[\frac{g+3}{4}\right]} (2+\alpha)(1+\alpha) \#N_{g}(\alpha).$$

In particular, for $g=1$, we have 
$\#N_{1}(0)=0$ and $\#N_{1}(1)=1$; so $\#X_{1}=3$.

\begin{lemm}
If $0 \leq \alpha \leq \left[\frac{g+3}{4}\right]$, $l_{0}=1$ and, for $\alpha>0$, $l_{\alpha}=0$, then

$$ \#N_{g}(\alpha)=\left\{ \begin{array}{ll}
\displaystyle{
\sum_{l=1}^{\frac{\left[\frac{g+3-4\alpha}{3}\right]+1}{2}} \left(1+\left[\frac{g+6-4\alpha-6l)}{4}\right] \right)}, & \mbox{if $g \geq 2$ even and $\left[\frac{g+3-4\alpha}{3}\right]$ odd.}\\
\\
\displaystyle{
\sum_{l=1}^{\frac{\left[\frac{g+3-4\alpha}{3}\right]}{2}} \left(1+\left[\frac{g+6-4\alpha-6l)}{4}\right] \right)}, & \mbox{if $g \geq 2$ even and $\left[\frac{g+3-4\alpha}{3}\right]$ even.}\\
\\
\displaystyle{
\sum_{l=l_{\alpha}}^{\frac{\left[\frac{g+3-4\alpha}{3}\right]}{2}} \left(1+\left[\frac{g+3-4\alpha-6l)}{4}\right] \right)}, & \mbox{if $g \geq 3$ odd and $\left[\frac{g+3-4\alpha}{3}\right]$ even.}\\
\\
\displaystyle{
\sum_{l=l_{\alpha}}^{\frac{\left[\frac{g+3-4\alpha}{3}\right]-1}{2}} \left(1+\left[\frac{g+3-4\alpha-6l)}{4}\right] \right)}, & \mbox{if $g \geq 3$ odd and $\left[\frac{g+3-4\alpha}{3}\right]$ odd.}
\end{array}
\right.
$$
\end{lemm}
\begin{proof}
Note that, for  $\alpha=0$ we need $\gamma \geq 1$ and, for $\alpha>0$, we have $\gamma \geq 0$.
As $g+3-4\alpha=2\beta+3\gamma+4\delta$, then
$$\gamma \in \left\{l_{\alpha},\ldots, \left[\frac{g+3-4\alpha}{3}\right] \right\}$$
$$g+3-3\gamma \; \mbox{must be even}.$$

For each such $\gamma$, we have a freedom in the choice of
$$\delta \in \left\{0,\ldots, \left[\frac{g+3-4\alpha-3\gamma}{4}\right] \right\}$$
and, for fixed $\delta$, the value of $\beta$ is determined as
$$\beta=\frac{g+3-4\alpha-3\gamma-4\delta}{2}.$$

If $g \geq 2$ is even, then $\gamma \geq 1$ is odd, so
$\gamma=2l-1$, where 
$$1 \leq l \leq \left\{\begin{array}{ll}
\left(\left[\frac{g+3-4\alpha}{3}\right]+1\right)/2, & \mbox{ if $\left[\frac{g+3-4\alpha}{3}\right]$ is odd}\\
\\
\left[\frac{g+3-4\alpha}{3}\right]/2, & \mbox{ if $\left[\frac{g+3-4\alpha}{3}\right]$ is even}
\end{array}
\right.
$$

If $g \geq 3$ is odd, then $\gamma \geq l_{\alpha}$ is even, so
$\gamma=2l$, where 
$$l_{\alpha} \leq l \leq \left\{\begin{array}{ll}
\left(\left[\frac{g+3-4\alpha}{3}\right]-1\right)/2, & \mbox{ if $\left[\frac{g+3-4\alpha}{3}\right]$ is odd}\\
\\
\left[\frac{g+3-4\alpha}{3}\right]/2, & \mbox{ if $\left[\frac{g+3-4\alpha}{3}\right]$ is even}
\end{array}
\right.
$$

\end{proof}

\subsection{Example 1: Extended ${\mathbb Z}_{4}$-Schotky groups of rank $1$}
If we take $g=1$, then $\#X_{1}=3$. In this case, the tree possibilities are given by 
$$(a_{1},a_{2},a_{3},a_{4},a_{5},a_{6}) \in \{(1,0,0,0,0,0), (0,0,0,0,1,0), (0,0,0,0,0,1)\}.$$

The tuple $(1,0,0,0,0,0)$ produces an ${\mathbb Z}_{4}$-extended Schottky group $K_{1}$ which, 
up to conjugation by a suitable M\"obius transformation, is generated by $A(z)=\lambda \overline{z}$ where $\lambda >1$. In this case, there is exactly one Schottky group $\Gamma_{1}$ in $K_{1}$ with $K_{1}/\Gamma_{1} \cong {\mathbb Z}_{4}$; which is generated by $A^{4}(z)=\lambda^{4}z$. The group $K_{1}$ induces on the handlebody $M_{1}={\mathbb H}^{3}/\Gamma_{1}$ an orientation-reversing isometry of order $4$ acting freely. 

The tuple $(0,0,0,0,0,1)$ produces a extended  ${\mathbb Z}_{4}$-Schottky group $K_{2}$ which, 
up to conjugation by a suitable M\"obius transformation, is generated by $A(z)=\lambda \overline{z}$ and $B(z)=-z$, where $\lambda >1$. In this case, there is exactly one Schottky group $\Gamma_{2}$ in $K_{2}$ with $K_{2}/\Gamma_{2} \cong {\mathbb Z}_{4}$; which is generated by $BA^{2}(z)=-\lambda^{2}z$. The group $K_{2}$ induces on the handlebody $M_{2}={\mathbb H}^{3}/\Gamma_{2}$ an orientation-reversing isometry of order $4$ whose locus of fixed points is a simple closed geodesic. 

The tuple $(0,0,0,0,1,0)$ produces a extended  ${\mathbb Z}_{4}$-Schottky group $K_{3}$ which, 
up to conjugation by a suitable M\"obius transformation, is generated by $A(z)=\lambda z$ and $B(z)=i/\overline{z}$, where $\lambda >1$. In this case, there are two Schottky groups $\Gamma_{3}$ in $K_{3}$ with $K_{3}/\Gamma_{3} \cong {\mathbb Z}_{4}$ (these are, repectively, generated by $A$ and by $B^{2}A$). The group $K_{3}$ induces on the handlebody $M_{3}={\mathbb H}^{3}/\Gamma_{3}$ an orientation-reversing isometry of order $4$ with only an isolated fixed points in the interior, but whose square has a simple closed geodesic as locus of fixed points. 


\subsection{Example 2: Extended ${\mathbb Z}_{4}$-Schotky groups of rank $2$}
 If $g=2$, then $\#X_{2}=1$. The only tuple is $(a_{1},a_{2},a_{3},a_{4},a_{5},a_{6})=(0,1,1,0,0,0)$. The produced  extended ${\mathbb Z}_{4}$-Schottky group $K$ is, up to conjugation by a suitable M\"obius transformation, generated by $A(z)=-1/\overline{z}$ and an elliptic transformation of order $2$, say $B$.
In this case, there is exactly one Schottky group $\Gamma$ in $K$ with $K/\Gamma \cong {\mathbb Z}_{4}$; which is generated by $A^{2}B$ and $A^{-1}BA^{-1}$. The group $K$ induces on the handlebody $M={\mathbb H}^{3}/\Gamma$ an orientation-reversing isometry of order $4$ on which it has exactly one simple closed geodesic as locus of fixed points and whose square is the hyperelliptic involution (then its locus of fixed points consists of exactly $3$ pairwise disjoint simple geodesic arcs. The closed Riemann surface $S$ uniformized by $\Gamma$ (that is, the conformal boundary of $M$) corresponds to an algebraic curve of the form
$$y^{2}=x(x^{2}-1)(x-b)(x+\overline{b}^{-1}),$$
where $b^{4} \neq 0,1$. The orientation-reversing automorphism of order $4$ is given by
$$\tau: \left\{ \begin{array}{ccc}
x & \mapsto & \dfrac{-1}{\overline x}\\
&&\\
y & \mapsto & \left( \dfrac{-b}{\overline b} \right)^{1/2}\dfrac{\overline{y}}{\overline{z}^{3}}
\end{array}
\right.
$$

\section{Connection with handlebodies}\label{handles}
Let $M$ be a handlebody of genus $g \geq 2$. The space that parametrizes marked Schottky structures on $M$ is the markd Schottky space and the moduli space of them (so we forget the marking) is the Schottky space.

If $H_{1}$ and $H_{2}$ are (finite) groups of homeomorphisms of $M$, then they are (weakly) {\it topologically equivalent} if there is an orientation-preserving self-homeomorphism $h:M \to M$ such that $H_{2}= f H_{1} f^{-1}$. In \cite{Bartos},  Bartoszy\'nska provided the topological classification  for $g=2$. In \cite{Ka-Mi}, Kalliongis-Miller characterized orientation-preserving finite group actions for $g \geq 2$ and in \cite{Ka-Mc} Kalliongis-McCullough considered a topological picture of orientation reversing involutions. In all of these papers, the used method is combinatorial and $3$-dimensional in nature. In \cite{Ka-Mi} an explicit formulae to obtain the number of equivalence classes for cyclic groups of prime order $p$ was provided. In \cite{Costa-Mc} it was studied the case of free fixed point orientation-reversing group actions on handlebodies and a classification theorem was obtained in terms of algebraic invariants that involve Nielsen equivalence.

If $\Gamma$ is a Schottky group of rank $g$, then $M_{\Gamma}=({\mathbb H}^{3} \cup \Omega)/\Gamma$, where $\Omega$ is the region of discontinuity of $\Gamma$ is homeomorphic to $M$. We say that $\Gamma$ induces a {\it Schottky structure} on $M$ (it induces a complete hyperbolic structure on the interior of $M$, whose injectivity radius is bounded away from zero, and also it provides a Riemann surface structure on the topological boundary of $M$). A {\it conformal automorphism} (respectively, {\it anticonformal automorphism}) of $M_{\Gamma}$ is an orientation-preserving (respectively, orientation-reversing) homeomorphism $f:M_{\Gamma} \to M_{\Gamma}$ whose restriction to the interior hyperbolic $3$-manifold $M^{0}_{\Gamma}={\mathbb H}^{3}/\Gamma$ is an isometry.

The following well known fact permits to see the relation between finite groups of homeomorphisms of handlebodies and Schottky groups.

\begin{lemm}\label{notita}
Let $H$ be a finite group of homeomorphism of the handlebody $M$. Then there is a Schottky structure on $M$ and there is a group $\widehat{H} \cong H$ of conformal/anticonformal automorphisms of $M$ (with respect to the Schottky structure) which is homotopic to $H$.
\end{lemm}
\begin{proof}
Let $M^{0}$ be the interior of $M$ and let $X$ its boundary (a closed orientable surface of genus $g$).
The group $H$ acts on $X$ by restriction as a group $H_{X} \cong H$ of homeomorphisms. As a consequence of Nielsen's realization theorem \cite{Kerckhoff},  $X$ has a Riemann surface structure $S$ such that (up to homotopy equivalence) the group $H_{X}$ acts as a group of automorphisms on $S$. The handlebody $M$ provides a Schottky unifromization $(\Omega,\Gamma,P:\Omega \to S)$  with the property that $H_{X}<{\rm Aut}(S)$ lifts to a group $K$ of automorphisms of $\Omega$ and $M_{\Gamma}= M$.
 We observe that  $K$ is a subgroup of $\widehat{\mathbb M}$ with $\Gamma \lhd K$ and $H_{X} \cong K/\Gamma$. The group $K$ induces a group $\widehat{H}$ of conformal/anticonformal automorphisms of $M_{\Gamma}$, isomorphic to $H$, which is homotopic to $H$ on the boundary $S$. As $M$ is a compression body, $\widehat{H}$ and $H$ are homotopic on $M$.
\end{proof}

\subsection{The cyclic case}
Let us assume that $\tau:M \to M$ is a finite order homeomorphism and set $H=\langle \tau \rangle$. Let the order of $\tau$ be $n$ (respectively, $2n$) if $\tau$ is orientation-preserving (respectively, orientation-reversing). By Lemma \ref{notita}, there is a Schottky structure on $M$ for which we may though of $\tau$ (up to homotopy) as a conformal/anticonformal automorphism. Let us denote by $\Gamma$ the Schottky group which provides such Schottky structure on $M$. The results in \cite{GH} and Theorem \ref{clasifica1} complement the work done by Kalliongis-Miller in \cite{Ka-Mi} as follows.  By lifting $\tau$ to the universal cover space ${\mathbb H}^{3} \cup \Omega$, we obtain a group $K$ which is a
${\mathbb Z}_{n}$-Schottky group of rank $g$ if $\tau$ is orientation-preserving (respectively, an extended ${\mathbb Z}_{n}$-Schottky group of rank $g$ if $\tau$ is orientation-reversing)  with $\Gamma \lhd K$ and $K/\Gamma =\langle \tau \rangle$.

\subsubsection{Case $\tau$ is conformal}
In this case, as $K$ is a ${\mathbb Z}_{n}$-Schottky group of rank $g$, the structural decomposition provided in \cite{GH} asserts that $K$ can be constructed using $``a"$ cyclic groups, each one generated by a loxodromic transformation, $``b"$ cyclic groups generated by elliptic transformations and $``m"$ Abelian groups. Assume the $b$ elliptic cyclic groups have orders $n_{1}$,..., $n_{b}$, and that the $m$ Abelian groups are isomorphic to ${\mathbb Z}_{l_1} \oplus {\mathbb Z}$,..., ${\mathbb Z}_{l_m} \oplus {\mathbb Z}$. With this information we are able to describe the locus of fixed points as $\tau$ and the quotient orbifold $M/H$ as follows.
\begin{enumerate}
\item The locus of fixed points of the non-trivial powers of $\tau$ is given by a pairwise collection of $\sum_{j=1}^{m}n/l_{j}$ simple loops (closed geodesics of $M^{0}_{\Gamma}$) and $\sum_{j=1}^{b} n/n_j$ simple arcs connecting two different points on $S$ (the interiors of these arcs being simple geodesics of $M^{0}_{\Gamma}$). 

\item The quotient orbifold $M/H$ is a (topological) handlebody of genus $b+m$ whose conical locus consists of exactly $``b"$ simple arcs and $``m"$ simple loops; all of them disjoint.

\end{enumerate}

If $n=2$, then in \cite{Ka-Mi} it was noted that the number of topologically non-equivalent actions is $(g+2)(g+4)/8$ when $g$ is even and $(g+3)(g+5)/8$ when $g$ is odd. This number is the same as the number of topological different conjugacy classes of  ${\mathbb Z}_{2}$-Schottky groups of a fixed rank $g$. This is consequence of the fact that given any index two Schottky subgroups, say $\Gamma_{1}$ and $\Gamma_{2}$, of the same ${\mathbb Z}_{2}$-Schottky group $K$, then one may construct an orientation-preserving homeomorphism (even a quasi-conformal one) $f:\widehat{\mathbb C} \to \widehat{\mathbb C}$ so that $fKf^{-1}=K$ and $f\Gamma_{1}f^{-1}=\Gamma_{2}$.

\subsubsection{Case $\tau$ is anticonformal}
In this case $K$ is a extended ${\mathbb Z}_{2n}$-Schottky group of rank $g$, so Theorem \ref{clasifica1} asserts that $K$ can be constructed using groups of types (T0)--T(8). 
The basic groups of types (T0) and (T1) do not produce fixed points. 

The description of the connected components of fixed points of $H$ is 
as follows.

\begin{enumerate}
\item A group of type (T2), say generated by an elliptic transformation of order $d \geq 2$ a divisor of $n$, produces a simple arc as a conical component of $M/H$  (with the exception of its ends points, its lies in the interior of $M/H$). That arc has conical order equal to $d$. By lifting to $M$ such a conical arc, we obtain a collection of $2n/d$
simple arcs in $M$ (with the exception of its ends points, they lie in the interior). Each of these arcs is a connected component of the locus of fixed points of $\tau^{2n/d}$.

\item A group of type (T3), say generated by a pseudo-elliptic transformation of order $2$ (that is, an imaginary reflection) produces an isolated conical point in the interior of $M/H$ of conical order $2$. It liftings to $M$ provides a collection of $n$ points, being fixed points of $\tau^{n}$ (in this case, $n$ is necessarily odd and $\tau^{n}$ is an imaginary-reflection).

\item A group of type (T3), say generated by a pseudo-elliptic transformation of order $2d$, with $d \geq 2$ a divisor of $n$, produces a simple arc as a conical component of $M/H$. One of the end points belong to the interior of $M/H$ and the other to the boundary. The arc (with the interior end point deleted) has conical order $d$. The deleted point has conical order $2d$. By lifting such an arc, we obtain a collection of $n/d$ simple arcs as components of fixed points of $\tau^{2n/d}$. The lifting of the interior end point is a collection of $n/d$ fixed points of $\tau^{n/d}$ (in particular, this case  only happens if $n/d$ is odd).

\item A group of type (T4), say with the elliptic generator being of order $d \geq 2$ a divisor of $n$,  produces a simple loop as a conical component of $M/H$. Its conical order is $d$. By lifting such a loop, we obtain a collection of $2n/d$ simple loops as components of fixed points of $\tau^{2n/d}$.

\item A group of type (T5), say with the pseudo-elliptic generator of order $2$, produces two isolated conical point in the interior of $M/H$ of conical order $2$. Its liftings to $M$ provides a collection of $2n$ points, being fixed points of $\tau^{n}$ (in this case, $n$ is necessarily odd).

\item A group of type (T5), say that the pseudo-elliptic generator has order $2d$, with $d \geq 2$ a divisor of $n$, produces a simple arc as a conical component of $M/H$ of conical order $d$ (both of its end points has conical order $2d$). By lifting such an arc to $M$ we obtain a collection of $2n/d$ simple loops as components of fixed points of $\tau^{2n/d}$ and the lifted points (of the conical point of order $2d$) are isolated fixed points of $\tau^{n/d}$ (in which case, $n/d$ should be odd).

\item A group of type (T6) (then $n$ is even) produces a simple loop as a component of conical points of $M/H$, with conical order $2$. By lifting such a loop we obtain a collection of $n$ simple loops in $M$ as components of fixed points of the conformal involution $\tau^{n}$. 

\item A group of type (T7) (then $n$ is odd) produces a closed disc as a conical components of $M/H$ (the boundary of the disc belongs to the boundary of $M/H$ and the interior of the disc to the interior of $M/H$) of conical order $2$. By lifting such a disc, we obtain a collection of $n$ discs as components of fixed points of the anticonformal involution $\tau^{n}$.

\item A group of type (T8) (then $n$ is odd) produces a bordered compact orbifold as a component of conical points of $M/H$. By lifting it, we obtain a collection of bordered compact surfaces as components of fixed points of the anticonformal involution $\tau^{n}$.
\end{enumerate}

It follows that, in the case that $n$ is odd, $\tau$ has no dimension two real locus of fixed points if and only if  there are  no groups of types (T7) nor (T8) in the construction of $K$. In this case, the anticonformal involution $\tau^{n}$ only has isolated fixed points at most. We discuss this case in the example below.

\subsubsection{Example}
Let us assume $n$ is odd and there are no groups of type (T7) nor (T8) in the construction of $K$. As $n$ is odd, then neither we cannot use groups of type (T6). Assume that in the construction of $K$ we use
\begin{enumerate}
\item $a_{0}$ groups of type (T0);
\item $a_{1}$ groups of type (T1);
\item $a_{2}$ groups of type (T2), of orders $2 \leq l_{1} \leq l_{2} \leq \cdots \leq l_{a_{2}} \leq n$, where each $l_{j}$ is a divisor of $n$;
\item $a_{3}$ groups of type (T3), or orders $2 \leq 2r_{1} \leq 2r_{2} \leq \cdots \leq 2r_{a_{3}} \leq 2n$, where each $2r_{j}$ is a divisor of $2n$ but not a divisor of $n$;
\item $a_{4}$ groups of type (T4); and
\item $a_{5}$ groups of type (T5).
\end{enumerate}

It follows that ${\mathcal O}=\Omega/K$ is a closed Klein orbifold of signature
$$(2(a_{0}+a_{1})+a_{3}+2a_{4}+2a_{5};-;l_{1},l_{1},l_{2},l_{2},...,l_{a_{2}},l_{a_{2}},r_{1},r_{2},...,r_{a_{3}}),$$
that is, ${\mathcal O}$ is the connected sum of $2(a_{1}+a_{0})+a_{3}+2a_{4}+2a_{5}$ projective planes and $2a_{2}+a_{3}$ conical points (whose orders are given in the third part of the signature). 
In this way, ${\mathcal O}^{+}=\Omega/K^{+}$ is a Riemann orbifold of signature
$$(2(a_{0}+a_{1})+a_{3}+2a_{4}+2a_{5}-1;+;l_{1},l_{1},l_{1},l_{1},...,l_{a_{2}},l_{a_{2}},l_{a_{2}},l_{a_{2}},r_{1},r_{1},,...,r_{a_{3}},r_{a_{3}}),$$
that is, ${\mathcal O}^{+}$ is a closed Riemann surface of genus $2(a_{0}+a_{1})+a_{3}+2a_{4}+2a_{5}-1$ and $4a_{2}+2a_{3}$ conical points. 

Let $\Gamma$ be a Schottky group so that $\Gamma \lhd K$ and $K/\Gamma \cong {\mathbb Z}_{2n}$.
The closed Riemann surface $S=\Omega/\Gamma$ admits an anticonformal automorphism $\tau|_{S}$, of order $2n$, induced by $\tau$. Set $f=\tau^{2}$. Then, $f:S \to S$ is a conformal automorphism of order $n$, with $S/\langle f \rangle = {\mathcal O}^{+}$. It follows from Riemann-Hurwitz formula that the genus of $S$, that is, the rank of $\Gamma$ is
$$g=n \left( 2(a_{0}+a_{1})+a_{3}+2a_{4}+2a_{5}-2 +2\sum_{j=1}^{a_{2}}(1-l_{j}^{-1})+\sum_{l=1}^{a_{3}}(1-r_{l}^{-1})     \right) +1. $$


\end{document}